\documentclass[a4paper, reqno]{article}
\usepackage{amsmath, amsthm, amsfonts, amssymb}
\usepackage[applemac]{inputenc}
\usepackage[all]{xy}
\usepackage[english]{babel}
\usepackage[lite, initials]{amsrefs}
\usepackage[a4paper, total={125mm, 185mm}]{geometry}
\usepackage{authblk}
\newcommand{\forevery}[2]{\null\hfilneg\llap{$\forall#1\quad\qquad$}\hfil#2}
\let\de=\partial
\newcommand{\C}{\mathbb{C}}
\newcommand{\R}{\mathbb{R}}

\newcommand{\N}{\mathbb{N}}
\newcommand{\D}{\mathbb{D}}
\renewcommand{\H}{\mathbb{H}}
\let\phe=\varphi
\let\eps=\varepsilon

\newcommand{\Hol}{\mathop{\mathrm{Hol}}\nolimits}
\newcommand{\Aut}{\mathop{\mathrm{Aut}}\nolimits}
\newcommand{\id}{\mathop{\mathrm{id}}\nolimits}

\newcommand{\cancel}[2]{\ooalign{$\hfil#1/\hfil$\crcr$#1#2$}}
\newcommand{\void}{\mathord{\mathpalette\cancel{\mathrel{\hbox{\kern0pt\raise0.8pt\hbox
	{$\scriptstyle\bigcirc$}}}}}}
\newtheorem{Theorem}{Theorem}[section]
\newtheorem{Corollary}[Theorem]{Corollary}
\newtheorem{Proposition}[Theorem]{Proposition}

\theoremstyle{definition}
\newtheorem{Definition}[Theorem]{Definition}
\theoremstyle{remark}
\newtheorem{Remark}[Theorem]{Remark}
\theoremstyle{remark}
\newtheorem{Example}[Theorem]{Example}

\title{Random iteration on hyperbolic Riemann surfaces}
\author{Marco Abate\thanks{Partially supported by 2017 PRIN grant ``Real and Complex Manifolds: Topology, Geometry and Holomorphic Dynamics", Ministry of University and Research, Italy, and by 2020 PRA grant ``Sistemi dinamici in logica, geometria, fisica matematica e scienza delle costruzioni", University of Pisa, Italy.}}
\affil{Dipartimento di Matematica, Universit\`a di Pisa, Largo Pontecorvo 5, I-56127 Pisa, Italy. E-mail: marco.abate@unipi.it}
\author{Argyrios Christodoulou}
\affil{School of Mathematical Sciences, Queen Mary University of London, Mile End Road,
London E1 4NS, United Kingdom. E-mail: argyrios.christodoulou@qmul.ac.uk}
\date{September 19, 2021}

\begin{document}

\maketitle

\smallskip

\begin{abstract}
Let $\{f_\nu\}\subset\Hol(X,X)$ be a sequence of holomorphic self-maps of a hyperbolic Riemann surface~$X$. In this paper we shall study the asymptotic behavior of
the sequences obtained by iteratively left-composing or right-composing the maps~$\{f_\nu\}$; the sequences of self-maps of~$X$ so obtained are called left (respectively, right) iterated function systems. We shall obtain the analogue for left iterated function systems of the theorems proved by Beardon, Carne, Minda and Ng for right iterated function systems with value in a Bloch domain; and we shall extend to the setting of general hyperbolic Riemann surfaces results obtained by Short and the second author in the unit disk~$\D$ for iterated function systems generated by maps close enough to a given self-map.
\end{abstract}

\section{Introduction}
\label{sec:0}

A classical result in one variable holomorphic dynamics is the \emph{Wolff-Denjoy theorem:}

\begin{Theorem}[Wolff \cites{Wolff1926a, Wolff1926b, Wolff1926c}, Denjoy \cite{Denjoy1926}; 1926]
\label{th:0.WD}
Let $f\in\Hol(\D,\D)$ be a holomorphic self-map of the unit disk $\D\subset\C$, not an elliptic automorphism. Then there exists a point $\tau\in\overline{\D}$ such that the sequence $\{f^k\}$ of iterates of~$f$ converges, uniformly on compact subsets, to the constant~$\tau$.
\end{Theorem}

In this statement, an \emph{elliptic automorphism} is an automorphism of~$\D$ with a fixed point in~$\D$. Up to a conjugation an elliptic automorphism is a rotation, and so its dynamics is well understood; Theorem~\ref{th:0.WD} describes completely the dynamics of all other holomorphic self-maps of~$\D$. The point $\tau$ in the statement is the \emph{Wolff point} of the function~$f$.

This result has been generalised by Heins in 1941 to hyperbolic Riemann surfaces as follows:

\begin{Theorem}[Heins \cite{Heins1941a}; 1941]
\label{th:I.3.hypRiemDW}
Let $X$ be a hyperbolic Riemann surface, and 
let $f\in\Hol(X,X)$. Then either:
\begin{enumerate}
\item[\rm(i)] $f$ has an attracting fixed point in $X$, or
\item[\rm(ii)] $f$ is a periodic automorphism, or
\item[\rm(iii)] $f$ is a pseudoperiodic automorphism, or
\item[\rm(iv)] the sequence $\{f^k\}$ is compactly divergent.
\end{enumerate}
Furthermore, the case {\rm (iii)} can occur only if $X$~is either simply 
connected (and $f$~has a fixed point) or doubly connected (and $f$~has no 
fixed points).
\end{Theorem}

An \emph{attracting fixed point}~$z_0$ of~$f$ is a fixed point (that is, $f(z_0)=z_0$) where the derivative of $f$ (which is well-defined because $z_0$ is a fixed point) has modulus less than~1; in particular, it follows that the sequence of iterates $\{f^k\}$ converges, uniformly on compact subsets, to the constant function~$z_0$. A \emph{periodic automorphism} is an automorphism~$f\in\Aut(X)$ such that $f^q=\id_X$ for some $q\ge 1$ --- and then the dynamics of $f$ is trivial. A \emph{pseudoperiodic automorphism} is an automorphism $f\in\Aut(X)$ such that there exists a subsequence of iterates converging to the identity~$\id_X$. Elliptic automorphisms of~$\D$ are always periodic or pseudoperiodic. A doubly connected hyperbolic Riemann surface is necessarily biholomorphic to a pointed disk or to an annulus; furthermore,
pseudoperiodic automorphisms are conjugated to rotations, either of the disk or of the pointed disk or of an annulus, and so their dynamics is well known.

Finally, a sequence $\{f_\nu\}\subset\Hol(X,Y)$ of holomorphic maps between two Riemann surfaces is \emph{compactly divergent} if for every compact $K\subseteq X$ and every compact $L\subseteq Y$ there is a $\nu_0\in\N$ such that $f_\nu(K)\cap L=\void$ for all $\nu\ge\nu_0$. 
Roughly speaking, a compactly diverging sequence is diverging to infinity (or, more precisely, is converging to the infinity point of the Alexandroff compactification of~$Y$). Restricting to the case of self-maps for simplicity, when $X=Y=D$ is a domain in a larger compact Riemann surface~$\widehat X$ then a compactly divergent sequence of iterates is converging toward the boundary, in the sense that all accumulation points of the sequence are constant maps contained in the boundary, and the set of accumulation points is closed and connected. Furthermore, when the boundary of~$D$ is sufficiently nice then Heins has obtained a complete analogue of the Wolff-Denjoy theorem:

\begin{Theorem}[Heins \cite{Heins1988}; 1988]
\label{th:0.Heins}
Let $D\subset\widehat X$ be a hyperbolic domain in a compact Riemann surface~$\widehat X$. Assume that $\de D$ consists of a finite number of isolated points or disjoint Jordan curves. Let $f\in\Hol(D,D)$ be such that the sequence of iterates~$\{f^k\}$ is compactly divergent in~$D$.
Then the sequence $\{f^k\}$ converges, uniformly on compact subsets, to a point $\tau\in\de D$.
\end{Theorem}

See also~\cite{Abate} for proofs and other related results.

%
%
%

A sequence of iterates is obtained by composing the same map over and over. Stimulated from problems coming from continued fractions theory, computer simulations of dynamical systems and other sources, in the '80s mathematicians have started to study properties of sequences obtained by composing different self-maps. This area of exploration is often called ``random dynamics" because the self-maps to consider can be chosen at random according to a suitable probability distribution, for instance with mass concentrated around a fixed self-map, thus obtaining random perturbations of the sequence of iterates. We shall not pursue a probabilistic approach in this paper, preferring a more topological approach, but we kept the term ``random dynamics" in the title as done, for instance, in \cite{BeardonCarneMindaNg2004}.

Let us define the sequences of self-maps appearing in random dynamics.

\begin{Definition}
\label{def:I.3.ifs}
Let $\{f_\nu\}$ be a sequence of self-maps of a Riemann surface~$X$, all different from the identity.
The \emph{left} (or \emph{direct} or \emph{forward}) 
\emph{iterated function system} (or \emph{composition system}) generated by $\{f_\nu\}$ is the sequence of self-maps $\{L_\nu\}$ given by
\[
L_\nu=f_\nu\circ f_{\nu-1}\circ\cdots\circ f_0\;.
\] 
The \emph{right} (or \emph{reverse} or \emph{backward}) \emph{iterated function system} generated by $\{f_\nu\}$ is instead the sequence of self-maps $\{R_\nu\}$ given by
\[
R_\nu=f_0\circ f_1\circ\cdots\circ f_\nu\;.
\] 
When all the maps $f_\nu$ belong to a given family $\cal F$ of self-maps of $X$ we shall say that the corresponding left or right iterated function system is \emph{in}~$\cal F$.
\end{Definition} 

Clearly, general left or right iterated function systems can have very erratic behaviours. Moreover, left and right iterated function systems generated by the same sequence can behave very differently; consider for instance the case when all $f_\nu$'s are constant. So to get meaningful theorems one has somehow to restrict the class $\mathcal{F}$ of functions used to generate the iterated function systems.

In this paper we shall consider iterated function systems generated by two classes of self-maps: 
maps belonging to $\Hol(X,\Omega)$, where $\Omega\subset X$ is a Bloch subdomain of~$X$; and maps sufficiently close to a given self-map $F\in\Hol(X,X)$.

Roughly speaking (see Section~\ref{sec:1} for a precise definition) a Bloch subdomain~$\Omega$ of a hyperbolic Riemann surface is a domain $\Omega\subset X$ such that all holomorphic maps from~$X$ into~$\Omega$ are strict contraction with respect to the Poincar\'e distance of~$X$. Right iterated function systems in~$\Hol(X,\Omega)$ have been thoroughly studied by Beardon, Carne, Minda and Ng~\cite{BeardonCarneMindaNg2004} in~2004. Left iterated function systems, on the other hand, are not so well studied, with the exception of a few results due to Gill  in the unit disk (see \cites{Gill1991a, Gill2012, Gill2017}) and to Keen and Lakic on plane domains (see \cite[Section 11.2]{KeenLakic2007}). 

In this paper, along the lines followed in~\cite{BeardonCarneMindaNg2004} for right iterated function systems, we shall complete the study of left iterated function systems in Bloch domains proving the following result:

\begin{Theorem}
\label{th:1.uno}
Let $\Omega\subset X$ be a Bloch domain in a hyperbolic Riemann surface~$X$, and let $\{f_\nu\}$ be a sequence of holomorphic self-maps of~$X$ 
such that  $f_\nu(X)\subseteq\Omega$ for all~$\nu\in\N$, that is $\{f_\nu\}\subset\Hol(X,\Omega)$. Then:
\begin{enumerate}
\item[\rm (i)] all limit points of the left iterated function systems $\{L_\nu\}$ generated by~$\{f_\nu\}$ are constant;
\item[\rm (ii)] let $z_\nu\in\Omega$ be the unique fixed point of~$f_\nu$. Then $\{L_\nu\}$ converges to a constant $z_\infty\in X$ if and only if $z_\nu\to z_\infty$.
 \end{enumerate}
\end{Theorem}

Random iteration for holomorphic self-maps ot the unit disk $\D$ sufficiently close to a given map have been thoroughly studied by Short and the second author in 2019 (see~\cite{ChristodoulouShort2019}). In this paper we shall extend their results to the setting of generic hyperbolic Riemann surfaces in the spirit of Heins theorems. 
According to Theorem~\ref{th:0.Heins} we need to distinguish three cases: when $F\in\Hol(X,X)$ has an attracting fixed point; when the sequence $\{F^\nu\}$ of iterates of~$F$ is compactly divergent; and when $F$ is a periodic or pseudoperiodic automorphism of~$X$.

 When $F$ has an attracting fixed point we shall prove a fairly complete result:

\begin{Theorem}
\label{th:1.due}
Let $X$ be a hyperbolic Riemann surface and
let $F\in\Hol(X,X)$ be with an attracting fixed point $z_0\in X$. Then:
\begin{enumerate}
\item[\rm(i)] there exists a neighbourhood $\mathcal{U}$ of~$F$ in~$\Hol(X,X)$ such that every right iterated function system generated by $\{f_\nu\}\subset\mathcal{U}$ converges to a constant in~$X$;
\item[\rm(ii)]if $\{f_\nu\}\subset\Hol(X,X)$ is a sequence converging to~$F$ then the left iterated function system generated by $\{f_\nu\}$ converges to~$z_0$.
\end{enumerate}
\end{Theorem}

When the sequence of iterates of $F$ is compactly divergent, simple examples (see, e.g., Example~\ref{ex:I.3.convWDboundb}) show that in general we cannot deduce much on the behavior of right iterated function systems generated by functions close to~$F$. On the other hand, if $f_\nu$ converges to~$F$ fast enough we shall prove that the behavior of $\{L_\nu\}$ is dictated by the dynamical behavior of~$F$:

\begin{Theorem}
\label{th:1.tre}
Let $X$ be a hyperbolic Riemann surface, and let $F\in\Hol(X,X)$ be such that the sequence of iterates $\{F^\nu\}$ is compactly divergent. Then we can find a sequence 
of neighbourhoods $\mathcal{U}_\nu\subset\Hol(X,X)$ of~$F$ such that if $f_\nu\in\mathcal{U}_\nu$ for all~$\nu\in\N$ then the left iterated function system $\{L_\nu\}$ generated by~$\{f_\nu\}$ is compactly divergent. Furthermore, if $X\subset\widehat{X}$ is a hyperbolic domain in a compact Riemann surface and $\{F^\nu\}$ converges to a point $\tau\in\de X$ then 
$\{L_\nu\}$ converges to~$\tau$.
\end{Theorem}

Surprisingly enough, if the convergence of $f_\nu$ to~$F$ is too slow then the left iterated function system might be oscillating even when the sequence of iterates of $F$ converges to a point in the boundary; see Example~\ref{ex:I.3.convWDbound}.

Finally, if $F$ is a periodic or pseudoperiodic automorphism of~$X$, in Section~\ref{sec:3} we shall see that the only interesting case is $X=\D$, that has already been
studied in~\cite{ChristodoulouShort2019}.

\section{Preliminaries}
\label{sec:1}

In this section we collect known definitions and results that we shall use in the sequel.

Let $X$ be a hyperbolic Riemann surface. We shall denote by $\omega_X$ the Poincar\'e distance on~$X$, which is a complete distance whose main property is the classical Schwarz-Pick lemma:

\begin{Theorem}
\label{th:SPhyp}
Let $X$ and $Y$ be two  hyperbolic Riemann surfaces,
and $f\colon X\to Y$ a holomorphic function. Then
\[
\forevery{z_1,\,z_2\in X}{\omega_Y\bigl(f(z_1),f(z_2)\bigr)\le\omega_X
	(z_1,z_2)\;.}
\]
Furthermore, equality at some $z_1\ne z_2$ implies that $f$ is a covering map; conversely, if $f$ is a covering map then for every $w_1$, $w_2\in Y$ and $z_1\in X$ with $f(z_1)=w_1$ we can find $z_2\in X$ so that $f(z_2)=w_2$ and $\omega_Y(w_1,w_2)=\omega_X(z_1,z_2)$.
\end{Theorem}

For a proof see, e.g., \cite{BeardonMinda2007}.

For each $z\in X$ and $R>0$ we shall denote by $B_X(z,R)$ the ball with respect to~$\omega_X$ centred in~$z$ and with radius~$R$. The completeness of~$\omega_X$ 
implies that the closed balls $\overline{B_X(z,R)}$ are compact in~$X$.

\begin{Definition}
\label{def:I.8.Bloch}
If $\Omega\subset X$ is a domain in a hyperbolic Riemann surface~$X$ and $z\in\Omega$ put
\[
R(z;\Omega,X)=\sup\{r>0\mid B_X(z,r)\subseteq\Omega\}
\]
and
\[
R(\Omega,X)=\sup_{z\in\Omega}R(z;\Omega,X)=\sup\{r>0\mid\exists z\in\Omega: B_X(z,r)\subseteq\Omega\}\;.
\]
We say that $\Omega$ is a \emph{Bloch domain} of~$X$ if $R(\Omega,X)<+\infty$. 
\end{Definition}

It is easy to see that relatively compact subdomains of a hyperbolic Riemann surface are Bloch domains; but it is not too difficult to find examples of Bloch domains which are not relatively compact (see, e.g., \cite{BeardonCarneMindaNg2004}).

The main property of Bloch domains is the following: 

\begin{Proposition}[Beardon, Carne, Minda, Ng \cite{BeardonCarneMindaNg2004}]
\label{th:I.3.Bloch}
Let $\Omega\subset X$ be a Bloch domain in a hyperbolic Riemann surface~$X$. Then there exists $0<\ell<1$ such that
\[
\omega_X\bigl(f(z),f(w)\bigr)\le\ell\, \omega_X(z,w)
\]
for all $z$,~$w\in X$ and all $f\in\Hol(X,\Omega)\subset\Hol(X,X)$. 
\end{Proposition}

A \emph{hyperbolic domain}~$D$ in a compact Riemann surface~$\widehat X$ is a domain $D\subset\widehat X$ which is a hyperbolic Riemann surface on its own.
A useful property of the Poincar\'e distance of hyperbolic domains is the following: 

\begin{Proposition}
\label{th:bP}
Let $D\subset\widehat X$ be a hyperbolic domain in a compact Riemann surface.
Take~$\tau_0\in\de D$, and a sequence~$\{z_\nu\}\subset D$ 
converging to~$\tau_0$. Let $\{w_\nu\}\subset D$ be another sequence, and 
assume there is~$M>0$ so that
\[
\forevery{\nu\in\N}{\omega_D(z_\nu,w_\nu)< M\;.}
\]
Then $w_\nu\to\tau_0$ as~$\nu\to+\infty$.
\end{Proposition}

For a proof see, e.g., \cite{Milnor2000}.

We shall need a few properties of the spaces of holomorphic maps between hyperbolic Riemann surfaces, that we shall always consider endowed with the compact-open topology (which is equivalent to the topology of uniform convergence on compact sets).
The first one is the classical Vitali theorem (for a proof see, e.g., \cite{Abate}*{Theorem 1.1.45}):

\begin{Theorem}[Vitali]
\label{th:Vitali}
Let $X$ and $Y$ be hyperbolic Riemann surfaces, and 
let $\{f_\nu\}$ be a sequence of functions in $\Hol(X,Y)$. Assume there is a 
set $A\subset X$ with at least one accumulation point such that $\{f_\nu(z)\}$ 
converges for every $z\in A$. Then $\{f_\nu\}$ converges uniformly on compact 
subsets of~$X$ to a function $f\in\Hol(X,Y)$.
\end{Theorem}

The next two results are again due to Heins~\cite{Heins1941b}:

\begin{Theorem}
\label{th:I.2.unocinque} 
Let $X$ be a hyperbolic Riemann surface 
with non-abelian fundamental group. Then $\id_X$~is isolated in~$\Hol(X,X)$.
In particular, $\Aut(X)$~is discrete, and each $f\in\Aut(X)$ is isolated in~$\Hol(X,X)$.
\end{Theorem}

\begin{Proposition}
\label{th:I.2.unootto} 
Let $X$ be a hyperbolic Riemann surface 
not biholomorphic to~$\D$ or~$\D^*=\D\setminus\{0\}$. Then $\Aut(X)$~is open and closed 
in~$\Hol(X,X)$.
\end{Proposition}

For a modern proof see, e.g., \cite{Abate}*{Theorem 1.2.19, Corollary~1.2.24}. 

The next result is another consequence of Heins' works; we include a proof for completeness.

\begin{Corollary}
\label{th:I.3.selfcov}
Let $X$ be a hyperbolic Riemann surface and $f\in\Hol(X,X)$ a self-covering of~$X$. Assume that $f$ has a fixed point~$z_0\in X$. Then $f$ is a periodic or pseudoperiodic automorphisms of~$X$, and $z_0$ is not attracting.
\end{Corollary}

\begin{proof}
Let $\pi_X\colon\D\to X$ the universal covering map.
Given $\tilde z_0\in\pi_X^{-1}(z_0)$ we can choose a lifting $\tilde f\colon\D\to\D$ of~$f$ such that $\tilde f(\tilde z_0)=\tilde z_0$. Since (see, e.g., \cite{Lee2000}*{Lemma~12.1})
$\tilde f$ is a self-covering of~$\D$, it must be an automorphism; therefore $|\tilde f'(\tilde z_0)|=1$, by the classical Schwarz-Pick lemma. This clearly implies that $|f'(z_0)|=1$, 
and the assertion follows from Theorem~\ref{th:I.3.hypRiemDW}.
\end{proof}

Finally, we recall a few standard facts about doubly connected hyperbolic Riemann surfaces; for proofs see, e.g., \cite{Abate}.

A Riemann surface is \emph{doubly connected} if and only if its fundamental group is isomorphic to~$\mathbb{Z}$. A doubly connected hyperbolic Riemann surface is necessarily biholomorphic either to the pointed disk~$\D^*=\D\setminus\{0\}$ or to an annulus 
\[
A(r,1)=\{z\in\C\mid r<|z|<1\}
\] 
for some $0<r<1$. A hyperbolic Riemann surface has abelian fundamental group if and only if it is simply connected (and hence is biholomorphic to~$\D$) or doubly connected.


Finally, we can completely describe the automorphisms of doubly connected hyperbolic Riemann surfaces. Indeed, every automorphism of $\D^*$ is of the form $\gamma(z)=e^{i \theta}z$ for some $\theta\in\R$; and every automorphism of $A(r,1)$ is either of the form 
$\gamma(z)=e^{i\theta}z$ or of the form $\gamma(z)=e^{i\theta}rz^{-1}$ for some 
$\theta\in\R$.

\section{Random iteration on Bloch domains}
\label{sec:2}

In this section we shall study left iterated function systems in $\Hol(X,\Omega)$, where $\Omega\subset X$ is a Bloch domain in a hyperbolic Riemann surface~$X$.
As in~\cite{BeardonCarneMindaNg2004} the main result is of topological nature:

\begin{Theorem}
\label{th:I.3.toprand}
Let $\{f_\nu\}\subset C^0(X,X)$ be a sequence of continuous self-maps of a metric space $(X,d)$, and assume that there exists $0<\ell<1$ such that
\begin{equation}
\forevery{x,y\in X\ \forall \nu\in\N}{d\bigl(f_\nu(x),f_\nu(y)\bigr)\le \ell d(x,y)\;.}
\label{eq:I.3.unifLip}
\end{equation}
Let $\{L_\nu\}$ be the left iterated function system generated by~$\{f_\nu\}$. Then
\begin{enumerate}
\item[\rm(i)] every (pointwise) limit point of $\{L_\nu\}$ is constant;
\item[\rm(ii)] if a subsequence $\{L_{\nu_k}\}$ converges pointwise to a constant $x_0\in X$ then it converges to~$x_0$ uniformly on compact subsets;
\item[\rm(iii)] assume that every $f_\nu$ has a (necessarily unique) fixed point $x_\nu\in X$. Then $\{L_\nu\}$ converges uniformly on compact subsets to a constant function if and only if the sequence $\{x_\nu\}$ converges in~$X$.
\end{enumerate}
\end{Theorem}

\begin{proof}
(i) Assume that a subsequence $\{L_{\nu_k}\}$ converges pointwise to a map $g\in C^0(X,X)$. Then for every $x$,~$y\in X$ we have
\[
d\bigl(g(x),g(y)\bigr)=\!\lim_{k\to+\infty}d\bigl(f_{\nu_k}\circ\cdots\circ f_0(x),f_{\nu_k}\circ\cdots\circ f_0(y)\bigr)\le\lim_{k\to+\infty}\ell^{\nu_k+1}d(x,y)=0\;,
\]
and thus $g$ is constant. 

(ii) Assume that $L_{\nu_k}(x)\to x_0$ for each $x\in X$. Let $K\subseteq X$ be compact, and let $d_0=\max_{x\in K}\{d(x_0,x)\}<+\infty$. Then
\[
\begin{aligned}
d\bigl(L_{\nu_k}(x),x_0)&\le d\bigl(L_{\nu_k}(x),L_{\nu_k}(x_0)\bigr)+d\bigl(L_{\nu_k}(x_0),x_0\bigr)\\
&\le \ell^{\nu_k+1} d(x,x_0)+d\bigl(L_\nu(x_0),x_0)\le \ell^{\nu_k+1} d_0+d\bigl(L_{\nu_k}(x_0),x_0)\;,
\end{aligned}
\]
for all $x\in K$, and thus $L_{\nu_k}\to x_0$ uniformly on~$K$. 

(iii) Assume first that $x_\nu\to x_\infty\in X$. Then for all $x\in X$ we have
\[
\begin{aligned}
d\bigl(L_\nu(x),\,x_\infty\bigr)&\le d\bigl(f_\nu\circ \cdots\circ f_0(x),f_\nu(x_\nu)\bigr)+d(x_\nu,x_\infty)\\
&\le \ell d\bigl(f_{\nu-1}\circ \cdots\circ f_0(x),x_\nu\bigr)+d(x_\nu,x_\infty)\\
&\le\cdots\\
\noalign{\vskip-10pt}
&\le\ell^{\nu+1}d(x,x_0)+\sum_{j=0}^{\nu-1}\ell^{\nu-j}d(x_j,x_{j+1})+d(x_\nu,x_\infty)\;.
\end{aligned}
\]
The first and third addends in the last line clearly goes to zero as $\nu\to+\infty$. To prove that the second addend goes to zero too, let $M=\max_j d(x_j,x_{j+1})<+\infty$, 
and fix $\eps>0$. Choose $\nu_0\in\N$ so that $d(x_j,x_{j+1})<\eps(1-\ell)/2$ as soon as $j\ge\nu_0$ and $\nu_1\in\N$ such that $\ell^j<\eps(1-\ell)/(2M)$ as soon as
$j\ge\nu_1$. Then if $\nu\ge\nu_0+\nu_1$ we have
\[
\begin{aligned}
\sum_{j=0}^{\nu-1}\ell^{\nu-j}d(x_j,x_{j+1})&=\sum_{j=0}^{\nu-\nu_1}\ell^{\nu-j}d(x_j,x_{j+1})+\sum_{j=\nu-\nu_1+1}^{\nu-1}\ell^{\nu-j}d(x_j,x_{j+1})\\
&\le M\sum_{j=0}^{\nu-\nu_1}\ell^{\nu-j}+\frac{\eps(1-\ell)}{2}\sum_{j=\nu-\nu_1+1}^{\nu-1}\ell^{\nu-j}\\
&=M\sum_{j=\nu_1}^\nu\ell^j+\frac{\eps(1-\ell)}{2}\sum_{j=1}^{\nu_1-1}\ell^j\\
&\le\frac{M}{1-\ell}\ell^{\nu_1}+\frac{\eps(1-\ell)}{2}\frac{1-\ell^{\nu_1}}{1-\ell}\le\eps\;.
\end{aligned}
\]
In this way we have proved that $L_\nu\to x_\infty$ pointwise --- and hence uniformly on compact subsets by (ii) --- as claimed.

Conversely, assume that $L_\nu\to x_\infty$ and, by contradiction, that there exist $r>0$ and a subsequence $\{x_{\nu_k}\}$ such that $d(x_{\nu_k},x_\infty)> r$
for all~$k$. Fix $x\in X$. Since $L_\nu(x)\to\infty$, we know that 
\[
d\bigl(L_\nu(x),x_\infty\bigr)<\frac{1-\ell}{1+\ell}r
\]
for $\nu$ large enough. Then for $k$ large enough we have
\[
\begin{aligned}
d(x_{\nu_k},x_\infty)-\frac{1-\ell}{1+\ell}r&\le
d(x_{\nu_k},x_\infty)-d\bigl(L_{\nu_k}(x),x_\infty)\\
&\le d\bigl(L_{\nu_k}(x),x_{\nu_k}\bigr)\\
&\le \ell d\bigl(L_{\nu_k-1}(x),x_{\nu_k}\bigr)\le\ell d\bigl(L_{\nu_k-1}(x),x_\infty\bigr)+\ell d(x_{\nu_k},x_\infty)\\
&\le \frac{1-\ell}{1+\ell}\ell r+\ell d(x_{\nu_k},x_\infty)\;,
\end{aligned}
\]
and then $d(x_{\nu_k},x_\infty)\le r$, contradiction.
\end{proof}

\begin{Remark}
\label{rem:I.3.topc}
If $(X,d)$ is complete the Banach fixed point theorem implies that every $f\in\Hol(X,X)$ satisfying \eqref{eq:I.3.unifLip} has a fixed point in~$X$, and thus the hypothesis in (iii) is automatically satisfied.
\end{Remark}

As recalled in Section~\ref{sec:1}, if $\Omega\subset X$ is a Bloch domain then there exists $0<\ell<1$ such that
\[
\omega_X\bigl(f(z),f(w)\bigr)\le \ell \omega_X(z,w)
\]
for every $z$,~$w\in X$ and every $f\in\Hol(X,\Omega)$. Therefore Theorem~\ref{th:1.uno} is a consequence of Theorem~\ref{th:I.3.toprand}, as shown in the following result, where for the sake of completeness we report also what has been proved in \cite{BeardonCarneMindaNg2004} for right iterated function systems:

%
%
%

\begin{Corollary}
\label{th:I.3.randituno}
Let $\Omega\subset X$ be a Bloch domain in a hyperbolic Riemann surface~$X$, and let $\{f_\nu\}$ be a sequence of holomorphic self-maps of~$X$ 
such that  $f_\nu(X)\subseteq\Omega$ for all~$\nu\in\N$, that is $\{f_\nu\}\subset\Hol(X,\Omega)$. Then:
\begin{enumerate}
\item[\rm(i)] all limit points of the left and right iterated function systems, $\{L_\nu\}$ and $\{R_\nu\}$, generated by~$\{f_\nu\}$ are constant;
\item[\rm(ii)] let $z_\nu\in\Omega$ be the unique fixed point of~$f_\nu$. Then $\{L_\nu\}$ converges to a constant $z_\infty\in X$ if and only if $z_\nu\to z_\infty$;
\item[\rm(iii)] if there is $z_0\in X$ such that the set $\{f_\nu(z_0)\}$ is relatively compact in~$X$ then $\{R_\nu\}$ converges to a constant function. 
 \end{enumerate}
\end{Corollary}

\begin{proof}
The statements for $\{R_\nu\}$ are in \cite{BeardonCarneMindaNg2004}. For $\{L_\nu\}$,
since the Poincar\'e distance of~$X$ is complete, by Proposition~\ref{th:I.3.Bloch} and Remark~\ref{rem:I.3.topc} we can apply Theorem~\ref{th:I.3.toprand}, and we are done.
\end{proof}

\section{Random iteration of small perturbations}
\label{sec:3}

In this section we shall discuss the behaviour of iterated function systems generated by functions close enough to a given self-map~$F$; in particular we would like to understand whether the dynamics of the iterated function systems mimics the dynamics of the sequence of iterates of~$F$. 

Recalling Theorem~\ref{th:I.3.hypRiemDW}, we see that we have three cases to consider: when $F$ has an attracting fixed point, when $F$ is a periodic or pseudoperiodic automorphism and when the sequence $\{F^k\}$ is compactly divergent. 

In the first case we have a fairly complete result. 

\begin{Theorem}
\label{th:I.3.convWDin}
Let $X$ be a hyperbolic Riemann surface and
let $F\in\Hol(X,X)$ be with an attracting fixed point $z_0\in X$. Then:
\begin{enumerate}
\item[\rm(i)] there exists a neighbourhood $\mathcal{U}$ of~$F$ in~$\Hol(X,X)$ such that every right iterated function system generated by $\{f_\nu\}\subset\mathcal{U}$ converges to a constant in~$X$;
\item[\rm(ii)]if $\{f_\nu\}\subset\Hol(X,X)$ is a sequence converging to~$F$ then the left iterated function system generated by $\{f_\nu\}$ converges to~$z_0$.
\end{enumerate}
\end{Theorem}

\begin{proof}
Fix $r>0$ and let $D=B_X(z_0,r)$. Since $\overline{D}$ is compact and $F$ is not a self-covering of~$X$ (see Corollary~\ref{th:I.3.selfcov}), by the general Schwarz-Pick lemma Theorem~\ref{th:SPhyp} there is $0<k<1$ such that
$\omega_X\bigl(F(z),F(w)\bigr)\le k\omega_X(z,w)$ for all $z$,~$w\in\overline{D}$. In particular, since $F(z_0)=z_0$, we have $F(\overline{D})\subseteq \overline{B_X(z_0,kr)}\subset D$. 
Choose a real number~$t$ such that $kr<t<r$ and put
\[
\mathcal{U}=\{h\in\Hol(X,X)\mid h(\overline{D})\subset B_X(z_0,t)\}\;;
\]
clearly $\cal U$ is a neighbourhood of~$F$. Notice that $D$ is a hyperbolic Riemann surface, and $B_X(z_0,t)$ is a Bloch domain, because it is relatively compact in~$D$.
If $\{f_\nu\}\subset\cal U$ then, by Corollary~\ref{th:I.3.randituno}.(iii), the corresponding right iterated function system converges in~$D$---and hence, by Vitali theorem (Theorem~\ref{th:Vitali}), in~$X$---to a constant contained in $\overline{B_X(z_0,t)}$, and thus (i) is proved.

For (ii), since $f_\nu\to F$ we have $f_\nu\in\cal U$ for all $\nu$ large enough; truncating $L_\nu$ by finitely many terms on the right and relabelling we can assume
without loss of generality that $f_\nu(\overline{D})\subset D$ for all $\nu\in\N$. Choose now $z\in\overline{D}$. We have $F^\nu(z)$,~$L_\nu(z)\in\overline{D}$ for all~$\nu\in\N$; hence
\[
\begin{aligned}
\omega_X\bigl(L_\nu(z),F^\nu(z)\bigr)&\le \omega_X\bigl(L_\nu(z),F\bigr(L_{\nu-1}(z)\bigr)\bigr)+\omega_X\bigl(F\bigr(L_{\nu-1}(z)\bigr),F^\nu(z)\bigr)\\
&\le\sup_{w\in\overline{D}}\omega_X\bigl(f_\nu(w),F(w)\bigr)+k\omega_X\bigl(L_{\nu-1}(z),F^{\nu-1}(z)\bigr)\;.
\end{aligned}
\]
Repeating this argument we get by induction 
\[
\omega_X\bigl(L_\nu(z),F^\nu(z)\bigr)\le\sum_{j=0}^\nu k^j\sup_{w\in\overline{D}}\omega_X\bigl(f_{\nu-j}(w),F(w)\bigr)\;.
\]
Fix $\eps>0$. Since $0<k<1$ and, by assumption, $f_\nu\to F$ uniformly on compact subsets, we can find $\nu_0$ large enough so that $k^{j}<\frac{(1-k)\eps}{4r}$ and $\sup\limits_{w\in\overline{D}}\omega_X\bigl(f_j(w),F(w)\bigr)<\frac{(1-k)\eps}{2}$ as soon as $j\ge\nu_0/2$. Therefore if $\nu\ge\nu_0$ we have
\[
\begin{aligned}
\omega_X\bigl(L_\nu(z),F^\nu(z)\bigr)&\le\sum_{j=0}^{\nu/2} k^j\sup_{w\in\overline{D}}\omega_X\bigl(f_{\nu-j}(w),F(w)\bigr)\\
&\quad+\sum_{j=\nu/2}^{\nu} k^j\sup_{w\in\overline{D}}\omega_X\bigl(f_{\nu-j}(w),F(w)\bigr)\\
&<\left[ \frac{(1-k)\eps}{2}+2rk^{\nu/2}\right]\sum_{j=0}^{\nu/2} k^j<2\frac{(1-k)\eps}{2}\frac{1}{1-k}=\eps\;,
\end{aligned}
\]
and so $\omega_X\bigl(L_\nu(z),F^\nu(z)\bigr)\to 0$ uniformly on~$\overline{D}$ as $\nu\to+\infty$. Since Theorem~\ref{th:I.3.hypRiemDW} implies $F^\nu\to z_0$, it follows that $L_\nu\to z_0$ on~$\overline{D}$ and hence, again by Vitali theorem, on~$X$, as claimed.
\end{proof}

\begin{Remark}
\label{rem:I.3.convWDin}
In Theorem~\ref{th:I.3.convWDin}.(i) by changing~$f_0$ one can change the constant limit of the right iterated function system~$\{R_\nu\}$. 
\end{Remark}

The next example, taken from~\cite{ChristodoulouShort2019}, shows that in Theorem~\ref{th:I.3.convWDin}.(ii) we cannot replace the hypothesis $f_\nu\to F$ by the hypothesis $\{f_\nu\}\subset\mathcal{U}$ 
for any neighbourhood $\mathcal{U}$ of~$F$.

\begin{Example}
\label{ex:I.3.convWDin}
Let $F\in\Hol(\D,\D)$ be given by $F(z)=\frac{1}{2}z$, and let $\mathcal{U}\subset\Hol(\D,\D)$ be a neighbourhood of~$F$. Given $\delta>0$ and $\theta\in\R$, put $f_{\delta,\theta}(z)=\frac{1}{2}z+\delta e^{i\theta}$. Clearly $f_{\delta,\theta}\in\Hol(\D,\D)$ as soon as $\delta<1/2$. 

We claim that we can choose $\delta$ small enough so that $f_{\delta,\theta}\in\mathcal{U}$ for all~$\theta\in\R$. Indeed, fix a compact subset $K\subset\D$ and let $V\subset\D$ be an open neighbourhood of $F(K)$, so that $F\in\mathcal{U}(K,V)=\{h\in\Hol(\D,\D)\mid h(K)\subset V\}$.
Since $F(K)\cap\de V=\void$, there is a $\delta_0>0$ such that $d\bigl(w,F(K)\bigr)<\delta_0$ implies $w\in V$, where $d$ denotes the Euclidean distance. As a consequence,
if $\delta\le\delta_0$ we have $f_{\delta,\theta}(K)\subset V$ for all~$\theta\in\R$. Since $\mathcal{U}$ contains a finite intersection of sets of the form $\mathcal{U}(K,V)$ the claim follows.

Given $\{\theta_\nu\}\subset\R$ it is easy to check by induction that the left iterated function system $\{L_\nu\}$ generated by $\{f_{\delta,\theta_\nu}\}$ is given by
\[
L_\nu(z)=\frac{1}{2^{\nu+1}}z+\delta\sum_{j=0}^\nu\frac{1}{2^{j}}e^{i\theta_{\nu-j}}\;.
\]
For instance, taking $e^{i\theta_j}=(-1)^j$ we get
\[
L_\nu(0)=\delta(-1)^\nu\sum_{j=0}^\nu\left(-\frac{1}{2}\right)^j
\]
that does not converge when $\nu\to+\infty$, and thus $\{L_\nu\}$ cannot be convergent. 
\end{Example}

The next case is when $\{F^\nu\}$ is compactly divergent. If $f_\nu$ converges to~$F$ fast enough then the dynamics of the left iterated function system generated by $\{f_\nu\}$ is dictated by the dynamics
of~$\{F^\nu\}$:

\begin{Theorem}
\label{th:I.3.convWDbound}
Let $X$ be a hyperbolic Riemann surface, and let $F\in\Hol(X,X)$ be such that the sequence of iterates $\{F^\nu\}$ is compactly divergent. Then we can find a sequence 
of neighbourhoods $\mathcal{U}_\nu\subset\Hol(X,X)$ of~$F$ such that if $f_\nu\in\mathcal{U}_\nu$ for all~$\nu\in\N$ then the left iterated function system $\{L_\nu\}$ generated by~$\{f_\nu\}$ is compactly divergent. Furthermore, if $X\subset\widehat{X}$ is a hyperbolic domain in a compact Riemann surface~$\widehat{X}$ and $\{F^\nu\}$ converges to a point $\tau\in\de X$ then 
$\{L_\nu\}$ converges to~$\tau$.
\end{Theorem}

\begin{proof}
Fix a reference point~$z_0\in X$. For $\nu\in\N$ set 
\[
D_\nu=\overline{B_X\bigl(z_0,1+\omega_X(F^{\nu}(z_0),z_0)\bigr)}\;.
\] 
Since $\omega_X\bigl(F^{\nu}(z_0),z_0\bigr)\to+\infty$ as $\nu\to+\infty$ we have $X=\bigcup\limits_{\nu\in\N} D_\nu$.
Given $\nu\in\N$, choose $z_1,\ldots,z_r\in D_\nu$ such that $F(D_\nu)\subseteq\bigcup\limits_{j=1}^r B_X\bigl(F(z_j),(3\cdot 2^{\nu+1})^{-1}\bigr)$.
Put $B_j=B_X\bigl(F(z_j),(3\cdot 2^{\nu+1})^{-1}\bigr)$, $K_j=D_\nu\cap F^{-1}(\overline{B_j})$ and $\widetilde{B_j}=B_X\bigl(F(z_j),2^{-\nu-2}\bigr)$; in particular, $D_\nu=\bigcup\limits_{j=1}^r K_j$ and if $z\in K_j$ 
then $\omega_X\bigl(F(z),F(z_j)\bigr)\le(3\cdot2^{\nu+1})^{-1}$. Finally, put
\[
\mathcal{U}_\nu=\bigcap_{j=1}^r \bigl\{h\in\Hol(X,X)\bigm| h(K_j)\subset\widetilde{B_j}\bigr\}\;;
\]
this is a neighbourhood of~$F$ because $F(K_j)\subseteq\overline{B_j}\subset\widetilde{B_j}$ for all $j=1,\ldots,r$ by construction. 

Take $h\in\mathcal{U}_\nu$. If $z\in D_\nu$ we must have $z\in K_j$ for some $j=1,\ldots,r$; then
\begin{equation}
\begin{aligned}
\omega_X\bigl(h(z),F(z)\bigr)&\le\omega_X\bigl(h(z),F(z_j)\bigr)+\omega_X\bigl(F(z_j),F(z)\bigr)\\
&<\frac{1}{2^{\nu+2}}+\frac{1}{3\cdot 2^{\nu+1}}<\frac{1}{2^{\nu+1}}\;.
\end{aligned}
\label{eq:I.3.convDW}
\end{equation}
Now take $\{f_\nu\}\subset\Hol(X,X)$ with $f_\nu\in\mathcal{U}_\nu$ for all~$\nu\in\N$; notice that \eqref{eq:I.3.convDW} implies that $f_\nu\to F$. We will prove,
by induction, that 
\[
\omega_X\bigl(L_\nu(z_0),F^{\nu+1}(z_0)\bigr)<1-\frac{1}{2^{\nu+1}}
\] 
for all $\nu\in\N$. For $\nu=0$ it follows immediately from \eqref{eq:I.3.convDW}. Assume it holds for $\nu-1$; then
\[
\begin{aligned}
\omega_X\bigl(L_\nu(z_0),&\,F^{\nu+1}(z_0)\bigr)\\
&\le\omega_X\bigl(L_\nu(z_0),F\bigl(L_{\nu-1}(z_0)\bigr)\bigr)+\omega_X\bigl(F\bigl(L_{\nu-1}(z_0)\bigr),F^{\nu+1}(z_0)\bigr)\\
&\le\omega_X\bigl(L_\nu(z_0),F\bigl(L_{\nu-1}(z_0)\bigr)\bigr)+\omega_X\bigl(L_{\nu-1}(z_0),F^{\nu}(z_0)\bigr)\\
&<\omega_X\bigl(L_\nu(z_0),F\bigl(L_{\nu-1}(z_0)\bigr)\bigr)+1-\frac{1}{2^\nu}\;.
\end{aligned}
\]
Now, since
\[
\begin{aligned}
\omega_X\bigl(L_{\nu-1}(z_0),z_0\bigr)&\le\omega_X\bigl(L_{\nu-1}(z_0),F^\nu(z_0)\bigr)+\omega_X\bigl(F^\nu(z_0),z_0\bigr)\\
&<1+\omega_X\bigl(F^\nu(z_0),z_0\bigr)\;,
\end{aligned}
\]
we have $L_{\nu-1}(z_0)\in D_\nu$. Therefore since $f_\nu\in\mathcal{U}_\nu$ we can use \eqref{eq:I.3.convDW} to get
\[
\omega_X\bigl(L_{\nu}(z_0),F\bigl(L_{\nu-1}(z_0)\bigr)\bigr)=\omega_X\bigl(f_\nu\bigl(L_{\nu-1}(z_0)\bigr),F\bigl(L_{\nu-1}(z_0)\bigr)\bigr)<\frac{1}{2^{\nu+1}}
\]
and thus
\[
\omega_X\bigl(L_\nu(z_0),\,F^{\nu+1}(z_0)\bigr)<\frac{1}{2^{\nu+1}}+1-\frac{1}{2^\nu}=1-\frac{1}{2^{\nu+1}}
\]
as claimed. 

As a consequence for any $z\in X$ we have
\[
\begin{aligned}
\omega_X\bigl(L_\nu(z),&\,z_0\bigr)\\
&\ge \omega_X\bigl(F^{\nu+1}(z_0),z_0\bigr)-\omega_X\bigl(F^{\nu+1}(z_0),L_\nu(z_0)\bigr)-\omega_X\bigl(L_\nu(z_0),L_\nu(z)\bigr)\\
&\ge \omega_X\bigl(F^{\nu+1}(z_0),z_0\bigr)-1-\omega_X(x_0,z)\to+\infty
\end{aligned}
\]
as $\nu\to+\infty$, and thus $\{L_\nu\}$ is compactly divergent. 

Finally, assume that $X\subset\widehat{X}$ is a hyperbolic domain and that $F^\nu\to\tau\in\de X$. Then since
$\omega_X\bigl(L_\nu(z_0),\,F^{\nu+1}(z_0)\bigr)<1$ for all $\nu\in\N$ we can apply Proposition~\ref{th:bP} to get $L_\nu(z_0)\to\tau$.
Moreover, for any $z\in D$ we also have $\omega_X\bigl(L_\nu(z),L_\nu(z_0)\bigr)\le\omega_X(z,z_0)$; so another application of Proposition~\ref{th:bP}
yields $L_\nu(z)\to \tau$ for all~$z\in X$, and we are done.
\end{proof}

Surprisingly enough, if the convergence of  $f_\nu$ to~$F$ is too slow then the left iterated function system might not be convergent, as the next example shows.

\begin{Example}
\label{ex:I.3.convWDbound}
Let $F\in\Hol(\H^+,\H^+)$ be given by $F(w)=w-1$. Moreover, for $\nu\in\N$ put 
\[
\phe_\nu(w)=\frac{\nu w-1}{w+\nu}=\frac{w\cos\theta_\nu-\sin\theta_\nu}{w\sin\theta_\nu+\cos\theta_\nu},
\]
where $\theta_0=\pi/2$ and $\theta_\nu=\arctan\frac{1}{\nu}$ for $\nu\ge 1$. It is easy to check that each $\phe_\nu$ is an elliptic automorphism of~$\H^+$ fixing~$i$; furthermore $\phe_\nu\to\id_{\H^+}$ as $\nu\to+\infty$. 

Now let $g_\nu=\phe_\nu\circ F\circ\phe^{-1}_\nu$. By construction, each $g_\nu$ is a parabolic automorphism of~$\H^+$; furthermore we have
\[
g_\nu(w)=\frac{(\nu^2+\nu+1)w-\nu^2}{w+\nu^2-\nu+1}\;,
\]
and a quick computation shows that $g_\nu(\nu)=\nu$. In particular, for each $w\in\H^+$ as soon as $k$ is sufficiently large we have $g_\nu^k(w)$ arbitrarily close to~$\nu$. Moreover, $g_\nu\to F$ as $\nu\to+\infty$.

We shall now build by induction a sequence $\{f_\nu\}\subset\Aut(\H^+)$ generating a left iterated function system $\{L_\nu\}$ and an increasing sequence $\{\nu_j\}\subset\N$ with the following properties:
\begin{itemize}
\item[(a)] $\nu_0=0$ and $\nu_1=1$;
\item[(b)] $f_0=g_0$ and $f_1=F$;
\item[(c)] $\nu_{2j}\ge\nu_{2j-1}+j$ for all $j\ge 1$;
\item[(d)] $f_{\nu_{2j-1}+1}=\cdots=f_{\nu_{2j}-j}=g_j$ and $f_{\nu_{2j}-j+1}=\cdots=f_{\nu_{2j+1}}=F$ for all $j\ge 1$;
\item[(e)] $|L_{\nu_{2j}}(i)|<1/2^j$ and $|L_{\nu_{2j+1}}(i)|>j$ for all $j\in\N$.
\end{itemize}
Since $|g_0(i)|<1$ and $|F\circ g_0(i)|>0$ condition (e) is satisfied for $j=0$. 
Now choose $n_1\in\N$ such that $|g_1^{n_1}\bigl(L_{\nu_1}(i)\bigr)-1|<1/2$; in particular, $\bigl|(F\circ g_1^{n_1})\bigl(L_{\nu_1}(i)\bigr)\bigr|<1/2$. Choose now $m_1\in\N$
so that $\bigl|(F^{m_1+1}\circ g_1^{n_1})\bigl(L_{\nu_1}(i)\bigr)\bigr|>1$; putting $\nu_2=\nu_1+n_1+1$, $\nu_3=\nu_2+m_1$, $f_{\nu_1+1}=\cdots=f_{\nu_2-1}=g_1$ and $f_{\nu_2}=\cdots=f_{\nu_3}=F$ we get $|L_{\nu_2}(i)|<1/2$ and $|L_{\nu_3}(i)|>2$, that is conditions (c)--(e) are satisfied for $j=1$.

Now given $j\ge 1$ assume by induction that we have found $\nu_0<\cdots<\nu_{2j-1}$ and $f_0,\ldots,f_{\nu_{2j-1}}\in\Aut(\H^+)$ satisfying (a)--(e). Choose $n_j\in\N$ such that
\[
|g_j^{n_j}\bigl(L_{\nu_{2j-1}}(i)\bigr)-j|<\frac{1}{2^j}\;;
\] 
in particular, $\bigl|(F^j\circ g_1^{n_j})\bigl(L_{\nu_{2j-1}}(i)\bigr)\bigr|<1/2^j$. Choose now $m_j\in\N$
so that $\bigl|(F^{m_j+j}\circ g_j^{n_j})\bigl(L_{\nu_{2j-1}}(i)\bigr)\bigr|>j$; putting $\nu_{2j}=\nu_{2j-1}+n_j+j$ and $\nu_{2j+1}=\nu_{2j}+m_j$ and choosing
$f_{\nu_{2j-1}+1},\ldots,f_{\nu_{2j+1}}$ as in (d) we get $|L_{\nu_{2j}}(i)|<1/2^j$ and $|L_{\nu_{2j+1}}(i)|>j$, as required.

In this way we have constructed a sequence $\{f_\nu\}\subset\Aut(\H^+)$ converging (very slowly) to~$F$ generating a left iterated function system with $L_{\nu_{2j}}(i)\to 0$
and $L_{\nu_{2j+1}}(i)\to\infty$ as $j\to+\infty$; in particular $\{L_\nu\}$ does not converge.
\end{Example}

The left iterated function system we constructed in this example, though not converging, it is still compactly divergent. It would be interesting to find an example (or to prove that it does not exists) of a self-map $F\in\Hol(\D,\D)$ such that $\{F^\nu\}$ is compactly divergent and of a sequence $\{f_\nu\}$ converging to~$F$ so that $\{L_\nu\}$ is not compactly divergent. 

Finally, there is no hope to get a version of Theorem~\ref{th:I.3.convWDbound} for right iterated function systems:

\begin{Example}
\label{ex:I.3.convWDboundb}
Let $F\in\Hol(\H^+,\H^+)$ be given by $F(w)=w+1$, and define $\{f_\nu\}\subset\Hol(\H^+,\H^+)$ by setting $f_0(w)=i+e^{2\pi i z}$ and $f_\nu=F$ for $\nu\ge 1$. Then 
$f_\nu\to F$ in the fastest possible way but $R_\nu=f_0$ for all $\nu\in\N$, and thus $\{R_\nu\}$ is not even compactly divergent.
\end{Example}

We are left with the case when $F$ is a periodic or pseudoperiodic automorphism of~$X$. If the fundamental group of~$X$ is not abelian, Theorem~\ref{th:I.2.unocinque} implies that $F$ is isolated in~$\Hol(X,X)$; therefore any map close enough to~$F$ coincides with~$F$ and the study of random iteration of functions sufficiently close to~$F$ reduces to the study of the dynamics of~$F$, which is trivial. 

If the fundamental group of~$X$ is abelian, we know that $X$ is biholomorphic either to~$\D$ or to~$\D^*$ or to an annulus~$A(r,1)$ with $0<r<1$.
If $X$ is biholomorphic to~$\D^*$ then every holomorphic self-map of~$\D^*$ extends to a holomorphic self-map of~$\D$, and thus random iteration on~$\D^*$ reduces to random iteration on~$\D$. 

If $X$ is biholomorphic to an annulus $A(r,1)$ then Proposition~\ref{th:I.2.unootto} says that $\Aut(X)$ is open in~$\Hol(X,X)$; in particular, maps sufficiently close to~$F$ are automorphisms of~$X$. The group $\Aut\bigl(A(r,1)\bigr)$ has two connected components, $A_1=\{\phe_{\theta,1}\mid \theta\in\R\}$
and $A_{-1}=\{\phe_{\theta,-1}\mid \theta\in\R\}$, where $\phe_{\theta,1}(z)=e^{i\theta}z$ and $\phe_{\theta,-1}(z)=e^{i\theta}rz^{-1}$.
If $F\in A_1$ then every holomorphic self-map of~$A(r.1)$ sufficiently close to~$F$ belongs to~$A_1$; since $A_1\subset\Aut(\D)$, in this case random iteration of holomorphic self-maps close to~$F$ is reduced to random iteration on~$\D$. If instead $F\in A_{-1}$ then every holomorphic self-map of~$A(r.1)$ sufficiently close to~$F$ belongs to~$A_{-1}$. Since it is easy to check that
\[
\phe_{\theta,-1}\circ\phe_{\eta,-1}=\phe_{\theta-\eta, 1}\;,\ \phe_{\theta,1}\circ\phe_{\eta,-1}=\phe_{\theta+\eta,-1}\ \hbox{and}\ 
 \phe_{\theta,-1}\circ\phe_{\eta,1}=\phe_{\theta-\eta,-1}\;,
 \]
we see that every iterated function system generated by self-maps close enough to~$F$ splits in the union of an iterated function system contained in~$A_1$, obtained considering an even number of maps, and of the composition of $\phe_{0,-1}$ with an iterated function system again contained in~$A_1$, obtained considering an odd number of maps. Thus in this case too we are led to the study of random iteration in~$\D$.

Summing up, when $F$ is a periodic or pseudoperiodic automorphism of~$X$ for our aims we can safely assume that $X=\D$. This situation has been studied in~\cite{ChristodoulouShort2019}; for the sake of completeness we conclude this paper by reporting the main result in this case, referring to \cite{ChristodoulouShort2019}
for proofs, examples and comments.

\begin{Theorem}
\label{th:I.2.convell}
Let $F\in\Aut(\D)$ be a periodic or pseudoperiodic (and hence elliptic) automorphism of~$\D$; the case $F=\id_\D$ is allowed. Let $\{f_\nu\}\subset\Hol(\D,\D)$ be a sequence of non-constant holomorphic self-maps
of~$\D$ for which 
\begin{equation}
\sum_{\nu=0}^\infty \omega\bigl(f_\nu(a),F(a)\bigr)<+\infty\quad\hbox{and}\quad \sum_{\nu=0}^\infty \omega\bigl(f_\nu(b),F(b)\bigr)<+\infty
\label{eq:I.3.convell}
\end{equation}
for two distinct points $a$, $b\in\D$. Then $f_\nu\to F$ as $\nu\to+\infty$, and the sequences $\{F^{-\nu}\circ L_\nu\}$ and $\{R_\nu\circ F^{-\nu}\}$ converge to non-constant
holomorphic self-maps of~$\D$. 
\end{Theorem}

\bigbreak

\footnotesize
\noindent 2020 Mathematics Subject Classification: 37H12 (primary);  37F99, 30D05 (secondary).

\noindent\textit{Keywords:} Wolff-Denjoy theorem; random iteration; iterated function system; Bloch domain.

\end{document}